\documentclass[sigconf]{acmart}

\AtBeginDocument{%
  \providecommand\BibTeX{{%
    \normalfont B\kern-0.5em{\scshape i\kern-0.25em b}\kern-0.8em\TeX}}}

\usepackage{amsmath}
\usepackage{amsthm}
\newtheorem{remark}{remark}
\usepackage{algorithm}
\usepackage{algpseudocode}
\usepackage{textcomp}
\usepackage{multirow}

\begin{document}

\title{Tight Bounds for the Maximum Distance Over a Polytope to a Given Point}

\author{Marius Costandin}
\email{costandinmarius@gmail.com}
\affiliation{ \institution{General Digits} \country{Romania}}

\begin{abstract}
 In this paper we study the problem of maximizing the distance to a given point  $C_0$ over a polytope $\mathcal{P}$. Assuming that the polytope is circumscribed by a known ball we construct an intersection of balls which preserves the vertices of the polytope on the boundary of this ball, and show that the intersection of balls approximates the polytope arbitrarily well. Then, we use some known results regarding the maximization of distances to a given point over an intersection of balls to create a new polytope which preserves the maximizers to the original problem.  Next, a new intersection of balls is obtained in a similar fashion, and as such, after a finite number of iterations, we end up with an intersection of balls over which we can maximize the distance to the given point. The obtained distance is shown to be a non trivial upper bound to the original distance and it is computed as a fixed point of a certain univariate function. Tests are made with maximizing the distance to a random point over the unit hypercube up to dimension $n = 100$. Several detailed 2-d examples are also shown. As future work, given a polytope $\mathcal{P}$ and a ball which includes it, we investigate a method for obtaining a sequence of balls with decreasing radius which include the given polytope.
\end{abstract}

%



\maketitle

\section{Introduction}

In modern times, the systematic investigation of the geometry of the intersection of congruent balls (that is balls with equal radius) was started with the paper \cite{ballPoly1}. There are three books that survey some particular parts of the literature dealing with such intersections: \cite{bookballPoly1},  \cite{bookballPoly2}, and \cite{bookballPoly3}. For more general references, perhaps it is worth choosing from there.

In this paper relying on some results from literature \cite{funcos1} related to intersection of balls, we shall study the problem of maximizing the distance to a given point $C_0$ over a polytope. Both problems are known to be NP-Hard in general but maximizing the distance over an intersection of balls allows a polynomial algorithm for some particular classes, as shown in \cite{funcos1} namely for the cases where $C_0$, the given point, is outside the convex hull of the balls centers. 

 Very shortly, we show in the following that the Subset Sum Problem can be written as such a distance maximization problem over a polytope, making this problem NP-Hard.

Indeed briefly, as presented in \cite{funcos1} let $n \in \mathbb{N}$ and consider $S \in \mathbb{R}^n$ and $T \in \mathbb{R}$. The associated subset sum problem, SSP(S,T) asks if exists $x \in \{0,1\}^{n}$ such that $x^T\cdot S = T$. For this, similar to \cite{sahni}, consider the optimization problem for $\beta > 0$:
\begin{align}\label{E3.13a}
& \max x^T\cdot(x - 1_{n \times 1}) + \beta \cdot S^T \cdot x \nonumber \\
& \hspace{0.5cm} \text{s.t} \ \ \   x \in \begin{cases} S^T\cdot x \leq T\\
0 \leq x_i \leq 1 \hspace{0.3cm} \forall i \in \{1, \hdots, n\}
\end{cases} 
\end{align} Let the feasible set be denoted by $\mathcal{P} = \{x \in \mathbb{R}^n| S^T\cdot x \leq T, 
0 \leq x_i \leq 1 \hspace{0.3cm} \forall i \in \{1, \hdots, n\} \}$. 

\begin{remark}\label{R1}
It is easy to see that the objective function is always smaller than or equal to $\beta \cdot T$. In fact the objective function reaches the value $\beta \cdot T$ if and only if the SSP(S,T) has a solution. 
\end{remark}

Note that the objective function can be rewritten as
\begin{align}\label{E2b}
&x^T\cdot x + \left(\beta \cdot S - 1_{n \times 1} \right)^T \cdot x = \nonumber \\
&= \left\| x - \frac{1_{n\times 1} - \beta \cdot S}{2} \right\|^2 - \left\| \frac{1_{n\times 1} - \beta \cdot S}{2} \right\|^2 \nonumber \\
& = \|x - C_0\|^2 - \|C_0\|^2
\end{align} with obvious definition for $C_0$. Since $C_0$ does not depend on $x$, we shall consider the optimization problem:
\begin{align}\label{E3a}
\max_{x \in \mathcal{P}} \|x - C_0\|^2
\end{align}  Using Remark \ref{R1}, note that the SSP has a solution iff (\ref{E2b}) is zero, that is the mximum distance in (\ref{E3a}) is $\|C_0\|^2$. The problem (\ref{E3a}) is a distance maximization over a polytope.  Indeed $\mathcal{P}$ is the intersection of the unit hypercube with the halfspace $\{ x | S^T\cdot x \leq T\}$. Any maximizer shall be located in a corner of the polytope $\mathcal{P}$. 

It is worth mentioning that throughout this paper we shall denote by $\mathcal{B}(C,R)$ the open ball centered at $C\in \mathbb{R}^n$ with radius $R > 0$ and with $\bar{\mathcal{B}}(C,R)$ the closed ball centered at $C\in \mathbb{R}^n$ with radius $R > 0$. We also denote by $1_{n\times 1}$ the vector in $\mathbb{R}^n$ where all entries are $1$ and sometimes we refer to $\mathbb{R}^n$ as $\mathbb{R}^{n \times 1}$ for $n \in \mathbb{N}$. 

This paper shall begin with the study the problem of maximizing the distance to a given point $C_0 \in \mathbb{R}^n$ over a polytope $\mathcal{P}$. 
\begin{align}\label{E1.1}
\max_{x \in \mathcal{P}} \|x- C_0\|
\end{align}  

For this, in \cite{funcos1} the authors considered a similar problem: the maximization of the distance to a given fixed point $C_0 \in \mathbb{R}^n$ over an intersection of balls with arbitrary radii.

 Let $m>n \in \mathbb{N}$ and $C_k \in \mathbb{R}^n$ for $k \in \{1, \hdots, m\}$ such that any facet of their convex hull does not contain more than $n$ points. 
For a fixed $C_0 \in \mathbb{R}^n$ and $r > 0$ consider 
\begin{align}\label{E1}
\mathcal{Q} &= \bigcap_{k = 1}^m \bar{\mathcal{B}}(C_k, r_k) \hspace{0.5cm} h(x) = \max_{k \in \{1, \hdots. m\}} \|x - C_k\|^2 - r_k^2 \nonumber \\
& g(x) =  \|x - C_0\|^2 \hspace{0.5cm} \mathcal{H}^{\star} = \mathop{\text{argmin}}_{h(x) \leq 1} h(x) - g(x)  \nonumber \\
& \mathcal{P}_{R^2} = \left\{ x \in \mathbb{R}^n \biggr| \max_{k \in \{1, \hdots. m\}} h(x) - g(x) \leq -R^2 \right\} 
\end{align} where $R>0$ and $\bar{\mathcal{B}}(y,R) = \{x \in \mathbb{R}^n | \|x - y\| \leq R \}$ denotes the closed ball of center $y$ and radius $R$. 

The problem studied in \cite{funcos1}, \cite{funcos2} is
\begin{align}\label{E2}
\max_{x \in \mathcal{Q}} \|x - C_0\|
\end{align}
The problem (\ref{E2}) is NP-Hard in general. Noting that $h(x) - g(x)$ is a piecewise linear function, follows that finding an element in $\mathcal{H}^{\star}$ is a convex optimization problem. 

The following results from \cite{funcos1} are reiterated:
\begin{enumerate}
\item The set $\mathcal{Q} = \{x | h(x) \leq 0\} \subseteq \mathcal{P}_{0^2}$

\item If $C_0 \in \text{int}(\text{conv} \{C_1, \hdots, C_m\})$ and the radii are equal then the set $\mathcal{H}^{\star}$ has exactly one element $x^{\star}$. This does not depend on the choice of $C_0$ and is the center of the minimum enclosing ball (MEB) of the points $C_1, \hdots, C_m$, see Theorem 2 in \cite{funcos1}. In this case $\mathcal{H}^{\star} \subseteq \mathcal{Q} $ and exists $\underline{R} > 0$ such that $\mathcal{H^{\star}} = \mathcal{P}_{\underline{R}^2}$. The Theorem 1 in  \cite{funcos1} states that
\begin{align}\label{E3}
\max_{x \in \mathcal{Q}} \|x - C_0\| = \min \{R > 0 | \mathcal{P}_{R^2} \subseteq \mathcal{Q}\}
\end{align} Basically, this means that as $R$ increases from $0$ to $\underline{R}$ the set $\mathcal{P}_{R^2}$ evolves from initially containing $Q$ to being included in $Q$. The parameter $R$ for which it first enters $Q$, is actually the maximum distance from $C_0$ to a point in $Q$. The extreme points will be the vertices of the polytope $\mathcal{P}_{R^2}$ last to enter the set $Q$, hence finitely many.  See Figure \ref{fig1} for a graphical representation of this case.

\item If the radii are equal, then from $C_0 \not\in \text{conv} \{C_1, \hdots, C_m\}$ follows that the set $\mathcal{P}_{R^2}$ is unbounded, for any $R \geq 0$. The Theorem 1 in \cite{funcos1} states that in this case
\begin{align}
\max_{x \in \mathcal{Q}} \|x - C_0\| = \max \{R > 0 | \mathcal{Q} \cap \mathcal{P}_{R^2} \neq \emptyset\}
\end{align}  Because $\mathcal{Q}$ is bounded and although unbounded $\mathcal{P}_{R^2}$ is shrinking as $R$ increases (being the level sets of $h(x) - g(x)$ one has $\mathcal{P}_{R_1^2} \subseteq \mathcal{P}_{R_2^2}$ for $R_1 \geq R_2$), follows that exists $R_0$ such that $\mathcal{P}_{R^2} \cap \mathcal{Q} = \emptyset$ for all $R > R_0$. Therefore the set $\mathcal{P}_{R^2}$ evolves from initially containing $\mathcal{Q}$ for $R = 0$ to not having common elements for $R > R_0$. The largest parameter $R$ for which the set $\mathcal{P}_{R^2}$ has common elements to $\mathcal{Q}$ is actually the maximum distance from $C_0$ to a point in $Q$. In this case it is proven that there is always an unique extreme point, see \cite{funcos1}. For this case it is possible to compute in polynomial time the maximum distance and the maximizer, as showed in \cite{funcos1}.

\item Finally, if $C_0 \in \partial \text{conv}(C_1, \hdots, C_m)$ then from \cite{funcos1} one has:
\begin{align}
\max_{x \in \mathcal{Q}} \|x - C_0\| & = \underline{R} = \max \{R > 0 | \mathcal{P}_{R^2} \neq \emptyset\} \nonumber \\
& = \|y - C_0\| \hspace{0.3cm} \forall y \in \partial \mathcal{Q} \cap \mathcal{H}^{\star}
\end{align} For this case it is proven in Theorem 1 from \cite{funcos2}, that if $\mathcal{Q} \subseteq \text{int}(\text{conv}(C_1, \hdots, C_m))$ and $C_0$ does not belong to an edge of any sort (an intersection of facets), then the maximizer is unique and is a vertex of $\mathcal{Q}$. 
\end{enumerate} 

\begin{figure}[h]
  \includegraphics[width = 8cm]{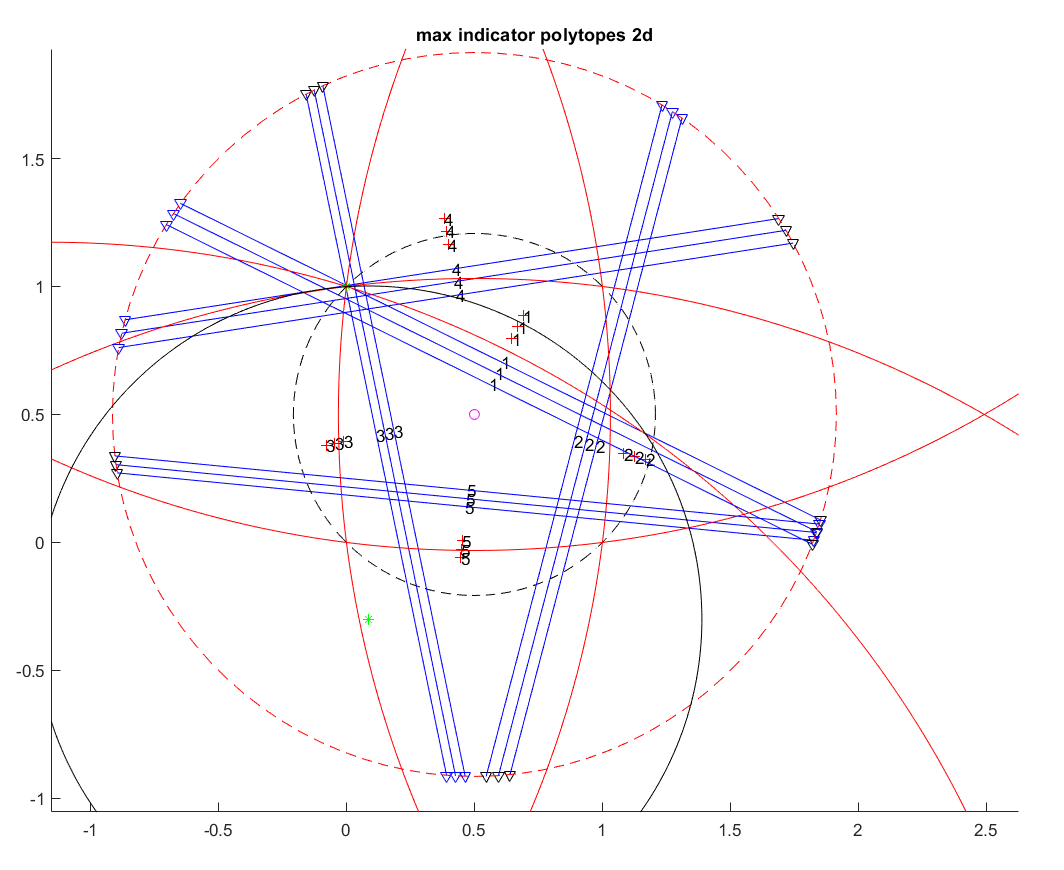}
  \caption{ The intersection of balls with red, the max indicator polytopes with blue: three instances of the family, $C_0$ is the green star and the identified maximum distance is given by the black circle. The last vertex to enter the intersection of balls is the farthest from $C_0$}
  \label{fig1}
\end{figure}

If the radii of the balls are equal, from \cite{funcos1}, follows that $\mathcal{H}^{\star} \subseteq \mathcal{Q} \iff C_0 \in \text{conv}(C_1, \hdots, C_m)$. However, even for arbitrary radii, if $C_0 \not\in \text{conv}(C_1, \hdots, C_m)$ then one can compute $\max_{x \in \mathcal{Q}} \|x - C_0\|$

%
%
%

\section{Maximizing the Distance over a Polytope}
Coming back to the problem of maximizing the distance to a given point $C_0 \in \mathbb{R}^n$ over a polytope $\mathcal{P}$. 
\begin{align}\label{E1.1}
\max_{x \in \mathcal{P}} \|x- C_0\|
\end{align}  we provide a polynomial algorithm which under the following assumptions:
\begin{enumerate}
\item A point $C_0 \neq C \in \mathbb{R}^n$ and $R>0$ are known such that 
\begin{enumerate}
\item \begin{align}
 \mathcal{P} \subseteq \mathcal{B}(C,R) 
\end{align}
\item 
\begin{align} \mathop{\text{argmax}}_{x \in \mathcal{P}} \|x- C_0\| \subseteq \partial \mathcal{B}(C,R)
\end{align}
\end{enumerate}
\item One has $C_0 - C$ is not orthogonal on any facet of $\mathcal{P}$
\end{enumerate}  is able (for some problems) to provide a maximizer and the maximum distance (i.e solve the problem), or in general provide an upper bound for the maximum distance.

For this we construct the intersection of balls $\mathcal{Q} \subseteq \mathcal{B}(C,R)$ such that 
\begin{align}
\mathop{\text{argmax}}_{x \in \mathcal{P}} \|x- C_0\| \leq \mathop{\text{argmax}}_{x \in \mathcal{Q}} \|x- C_0\|
\end{align} and $\mathcal{Q}$ preserves the corners of the polytope which are on the boundary of the ball $\mathcal{B}(C,R)$. 

Assume that
\begin{align} 
\mathcal{P} = \left\{ x \biggr| \begin{bmatrix} A_1^T\\ \vdots \\A_m^T \end{bmatrix}\cdot x +\begin{bmatrix} b_1\\ \vdots\\ b_m\end{bmatrix} \preceq  0_{m \times 1}\right\}
\end{align}

For each facet $i \in \{1, \hdots, m\}$ of the polytope, we construct a ball with center at $C_i$ and radius $r>0$ large enough, as such
\begin{align}
C_i = C + \rho_i\cdot \frac{A_i}{\|A_i\|}
\end{align} where $\rho_i \in \mathbb{R}$ and is computed as follows.
There are two methods we propose for this: one assures a constant radius for each ball, while the other method is simpler to work with in the following.
\begin{enumerate}
\item First method assures a constant radius for the balls: $r$. 
 We construct the point $P_i$ as the projection of the $C$ on the facet $i$ of $\mathcal{P}$. The value of $\|C - P_i\|$ can be computed as:
\begin{align}
\|C_i - P_i\|^2 = r^2 - \left( R^2 - \|C - P_i\|^2\right)
\end{align} and take $C_i$ accordingly. As such
\begin{align}
\rho_i - \frac{A_i^T\cdot C + b_i}{\|A_i\|} = \|C_i - P_i\|
\end{align}  Let 

\begin{align}
\mathcal{Q} = \bigcap_{k= 1}^m \mathcal{B}(C_i,r)
\end{align}
\item Second method: simpler to visualize even in higher dimensions, take $\rho_i$ constant, i.e. $\rho_i = \rho$. As such the radius of the balls are now varying from one to another. In order to compute the radius of the balls, now one first computes:
\begin{align}
\|C_i - P_i\| = \rho - \frac{A_i^T\cdot C + b_i}{\|A_i\|} 
\end{align} then 

\begin{align}
r^2_i  = \|C_i - P_i\|^2 - \left( R^2 - \|C - P_i\|^2\right)
\end{align}

\end{enumerate}

\begin{remark}\label{R1}
The above intersection of balls, assures that the boundary of each constructed ball leaves the same imprint on the boundary of the ball $\mathcal{B}(C,R)$ as the facet of $\mathcal{P}$. This is:
\begin{align}
\partial \mathcal{B}(C_i,r_i) \cap \partial \mathcal{B}(C,R) = \{x | A_i^T\cdot x + b_i = 0\} \cap  \partial \mathcal{B}(C,R)
\end{align} as such, $\mathcal{Q}$ is preserving the corners of $\mathcal{P}$ which are on the boundary of $\mathcal{B}(C,R)$. 
\end{remark}

Note that (for the case in which the radii are equal, but similar presentation can be made for the other case)
\begin{align}\label{E2.18}
\mathcal{P} \subseteq \mathcal{Q} \subseteq \bigcup_{x \in \mathcal{P}} \mathcal{B}(x, \epsilon)
\end{align} for all 
\begin{align} \label{E2.19}
\epsilon &\geq r - \|C_i - P_i\| = r - \sqrt{r^2 - (R^2 - \|C - P_i\|^2)} \nonumber \\
& = \frac{R^2 - \|C - P_i\|^2}{r + \sqrt{r^2 - (R^2 - \|C - P_i\|^2)}}
\end{align} 

Since $R^2 - \|C - P_i\|^2$ is fixed, the value of $\epsilon$ can be made arbitrarily small by increasing $r$. 

Because of the Remark \ref{R1} it is easy to see that 
\begin{align}
&\mathop{\text{argmax}}_{x \in \mathcal{P}} \|x - C_0\| \subseteq \{\text{corners of  } \nonumber \\
&\mathcal{P} \text{ on the boundary of }\partial \mathcal{B}(C,R)\} = \nonumber \\
& = \{\text{corners of  } \nonumber \\
&\mathcal{Q} \text{ on the boundary of }\partial \mathcal{B}(C,R)\}
\end{align} 

In the rest of the paper, unless otherwise explicitly stated, we shall consider the second method of constructing the balls, given a polytope inscribed in the ball $\mathcal{B}(C,R)$, i.e. $C_i = C + \rho \cdot \frac{A_i}{\|A_i\|}$. Note that as $\rho$ is increased the radii $r_i$ increase accordingly (hence the approximation can be made arbitrarily good for this method as well). 
 
\subsection{Algorithm Presentation}

Because (\ref{E2.18}) and (\ref{E2.19}) follows that exists $\rho$ large enough such that $C_0 \in \text{conv}(C_1, \hdots, C_m)$ and
\begin{align}
\mathop{\text{argmax}}_{x \in \mathcal{P}} \|x-C_0\| = \mathop{\text{argmax}}_{x \in \mathcal{Q}} \|x-C_0\|
\end{align}

For solving $\mathop{\text{argmax}}_{x \in \mathcal{Q}} \|x-C_0\|$ we apply Theorem 2 in \cite{funcos1} and construct the family of polytopes  $\mathcal{P}_{R^2}$. One particular member of this family,  namely $\mathcal{P}_{R_0^2}$ meets 
\begin{align}
&\mathop{\text{argmax}}_{x \in \mathcal{Q}} \|x - C_0\| \subseteq \{\text{corners of  } \nonumber \\
&\mathcal{P}_{R_0^2} \text{ on the boundary of }\partial \mathcal{B}(C,R)\}
\end{align} but $\mathcal{P}_{R_0^2} \subseteq \mathcal{Q} \subseteq \mathcal{B}(C,R)$. That is, the polytope $\mathcal{P}_{R_0^2}$ contains the maximizers of $\|x - C_0\|$ with $x \in \mathcal{Q}$. 

We propose now the following problem:
\begin{align}
	\max_{x \in \mathcal{P}_{R_0^2}} \|x - C_0\|
\end{align} It is easy to see that 
\begin{align}
	\mathop{\text{argmax}}_{x \in \mathcal{P}_{R_0^2}} \|x - C_0\| = \mathop{\text{argmax}}_{x \in \mathcal{Q}} \|x - C_0\| 
\end{align} Indeed, $\mathcal{P}_{R_0^2} \subseteq \mathcal{Q}$ and $\mathop{\text{argmax}}_{x \in \mathcal{Q}} \|x - C_0\|  \subseteq \{\text{corners of  }\mathcal{P}_{R_0^2}\}$

Since $\mathcal{P}_{R_0^2} \subseteq \mathcal{Q} \subseteq \mathcal{B}(C,R)$,  with the same method as above, obtain the intersection of balls $\mathcal{Q}^1_{R_0}$, from $\mathcal{P}_{R_0^2}^1 := \mathcal{P}_{R_0^2}$ in the same manner $\mathcal{Q}$ was obtained from $\mathcal{P}$, which as well preserves the corners of $\mathcal{P}_{R_0^2}$ on the boundary of $\mathcal{B}(C,R)$. The parameter $\rho$, with which the center of the balls forming $\mathcal{Q}_{R_0^2}^1$ are constructed, is maintained constant. 

The evolution of the this process is given by the following lemma.

\begin{lemma}\label{C1}
If $\rho$ is kept constant and $C_{k}^i - C$ is not co-linear with $C_0 - C$ for every $k \in \{1, \hdots, m\}$, than after a finite number of such steps, say $k$, one obtains that $C_0$ is no longer in the convex hull of the centers of balls forming the intersection of balls $\mathcal{Q}_{R_0^2}^k$. The number of these steps depend on the problem data.
\end{lemma}
\begin{proof}


Note that for each generation, $i$, of intersection of balls $\mathcal{Q}_{R_0^2}^i$, the centers of the balls are lying on the boundary of $\mathcal{B}(C,\rho)$. As seen in Figure \ref{fig2} with green (a 2-d example), their convex hull contains the point $C_0$, depicted with a triangle. 

\begin{figure}[h]
  \includegraphics[width = 8cm]{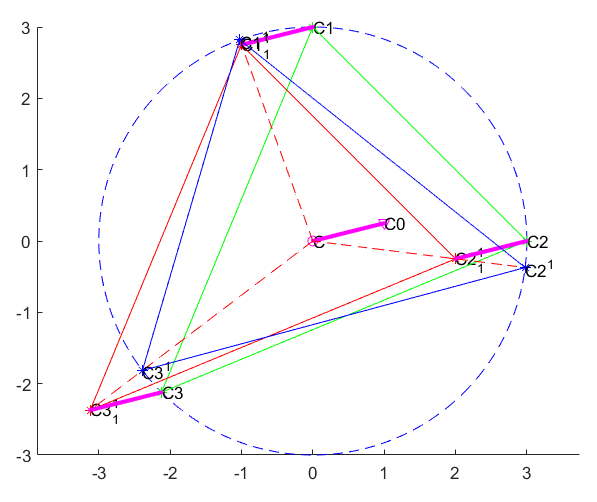}
  \caption{ An example: The initial convex hull with green. The translated convex hull with red. One of the facets approaches the point $C_0$. The adjusted convex hull with blue. Still has one facet closer to $C_0$ than in initial convex hull.}
  \label{fig2}
\end{figure}

Translating the convex hull of the center of balls forming $\mathcal{Q}_{R_0^2}^i$ by the vector $C - C_0$ one has at least one facet approaching the point $C_0$. See in Figure \ref{fig2} with red the translated original green convex hull. Since 
\begin{align}
C_{k}^{i+1} = C + \rho \cdot \frac{C_k^{i} + (C - C_0) - C}{\|C_k^{i} + (C - C_0) - C\|} 
\end{align} follows that the centers of balls in the next intersection of balls $\mathcal{Q}_{R_0^2}^{i+1}$ are obtained by projecting on the sphere $\partial \mathcal{B}(C,\rho)$ the translations of the centers of balls of the previous intersection of balls $\mathcal{Q}_{R_0^2}^i$. In the Figure \ref{fig2} these are the blue points forming the blue convex hull. Note that still one facet is closer to $C_0$ than in the initial green convex hull. The lemma states that after several such iterations, $C_0$ remains outside of the next convex hull.

Let $\theta_{k,i} = \angle(C_{0}-C, C_k^i-C) \Rightarrow \cos(\theta_{k,i}) =  \frac{(C_0 - C)^T\cdot (C_{k}^i - C)}{\|C_0 - C \|\cdot \|C_{k}^i - C\|}$ 
then for $C_{k}^{i+1} = C + \rho_k \cdot (C_k^i - C_0)$ with appropriate definition for the scalar $\rho_k$, one has
\begin{align}
\cos(\theta_{k,i+1}) &= \angle(C_{0}-C, C_{k}^{i+1}-C) = \frac{(C_{0}-C)^T\cdot ( C_{k}^{i+1}-C)}{\|C_{0}-C\| \cdot \|C_{k}^{i+1}-C\| } \nonumber \\
& = \rho_k\cdot \frac{(C-C_0)^T\cdot (C_k^i - C)}{\|C_{0}-C\| \cdot \|C_{k}^{i+1}-C\| } - \rho_k\cdot  \frac{\|C - C_0\|^2}{\|C_{0}-C\| \cdot \|C_{k}^{i+1}-C\| }  
\end{align} Since $\|C_{k}^{i+1}-C\| = \rho_k \cdot \|C_k^i - C_0\| = \rho_k \cdot \|C_k^i - C\|\cdot \frac{\|C_k^i - C_0\|}{\|C_k^i - C\|}$ follows that

\begin{align}
&\cos(\theta_{k,i+1}) = \cos(\theta_k^i) \cdot \frac{\|C_k^i - C\|}{\|C_k^i - C_0\|} - \frac{\|C - C_0\|}{ \|C_{k}^{i}-C_0\| } \nonumber \\
& = \cos(\theta_k^i) \cdot \frac{\|C_k^i - C\| - \|C_0 - C\| + \|C_0 - C\|}{\|(C_k^i -C) - (C_0- C)\|} - \frac{\|C - C_0\|}{ \|C_{k}^{i}-C_0\| } \nonumber \\
& = \cos(\theta_{k,i}) \cdot \frac{\|C_k^i - C\| - \|C_0 - C\| }{\|(C_k^i -C) - (C_0- C)\|} + (\cos(\theta_{k,i}) - 1) \cdot \frac{\|C - C_0\|}{ \|C_{k}^{i}-C_0\| }
\end{align}

 Assuming $\|C_{k}^i -C\| > \|C_0 - C\|$ and $0\leq \cos(\theta_{k,i}) < 1$, that is $C_k^i -C \neq (C_0-C) \cdot \alpha$ for any $\alpha \in \mathbb{R}$, one has  

\begin{align}\label{E63}
\cos(\theta_{k,i+1}) - \cos(\theta_{k,i}) &< (\cos(\theta_{k,i}) - 1) \cdot \frac{\|C - C_0\|}{ \|C_{k}^{i}-C_0\| } \nonumber \\
& < (\cos(\theta_{k,0}) - 1) \cdot \frac{\|C - C_0\|}{ V}
\end{align} where $\|C_{k}^{i}-C_0\| \leq V$, hence the cosine decreases at every iteration eventually reaching a small enough value to assert that $C_0$ is outside the convex hull of the balls centers, since it form a large angle with the vertices of the convex hull facet.

\end{proof}

Note the following inclusions for $i \leq k$
\begin{align}
\mathcal{P}_{R_0^2}^k \subseteq \mathcal{Q}_{R_0^2}^k \subseteq \bigcup_{x \in \mathcal{P}_{R_0^2}^k} \mathcal{B}(x,\epsilon)
\end{align} and as such for all $i \leq k - 1$
\begin{align}
\max_{x \in \mathcal{P}_{R_0^2}^i}\|x - C_0\| \leq \max_{x \in \mathcal{Q}_{R_0^2}^i} \|x - C_0\|= \max_{x \in \mathcal{P}_{R_0^2}^{i+1}} \|x - C_0\|
\end{align}

Finally, note that $\mathcal{Q}_{R_0^2}^i \subseteq \mathcal{B}(C,R)$. We can now give the main result of this paper:

\begin{corollary}\label{Co2.1}
For a given polytope $\mathcal{P}$, if a ball $\mathcal{B}(C,R)$ is known such that
\begin{enumerate}
\item  $\mathcal{P} \subseteq \mathcal{B}(C,R)$ 
\item  
\begin{align}
\mathcal{X}^{\star} = \mathop{\text{argmax}}_{x \in \mathcal{P}} \|x - C_0\| \subseteq \partial \mathcal{B}(C,R)
\end{align} i.e. the maximizers belongs to the boundary of the ball
\item A value $R_0 > 0$ is known such that $\|C_0 - x^{\star}\| = R_0$ 
\end{enumerate} then it is possible to obtain a set $\mathcal{Q}_{R_0^2}^k \subsetneq \mathcal{B}(C,R)$ such that one can solve in polynomial time 
\begin{align}
\{y^{\star}\} = \mathop{\text{argmax}}_{x \in \mathcal{Q}_{R_0^2}^k} \|x - C_0\|
\end{align}

 If $y^{\star} \in \partial \mathcal{B}(C,R)$ then $y^{\star} \in  \mathcal{X}^{\star}$ otherwise if $y^{\star} \in \mathcal{B}(C,R)$ then $\|y^{\star} - C_0\| > R_0$
\end{corollary}

\begin{proof}
Starting from $\mathcal{P}$ construct $\mathcal{Q}$ then $\mathcal{P}_{R_0^2}^1, \mathcal{Q}_{R_0^2}^1, \hdots$. The process ends, according to Lemma \ref{C1} with $\mathcal{Q}_{R_0^2}^k$ which is an intersection of balls with $C_0$ not in the convex hull of the balls centers. As such one can apply Theorem 2 in \cite{funcos1} to solve  $\{y^{\star}\} = \mathop{\text{argmax}}_{x \in \mathcal{Q}_{R_0^2}^k} \|x - C_0\|$. It is known that the number of maximizers in this situation is one. Because of the construction, the following are known:
\begin{enumerate}
\item 
\begin{align}
\|y^{\star} - C_0\| &= \max_{x \in \mathcal{Q}_{R_0^2}^k}\|x - C_0\| \geq \hdots  \geq \nonumber \\
& \geq \max_{x \in \mathcal{P}}\|x - C_0\| = \|x^{\star} - C_0\|
\end{align} for all $x^{\star} \in \mathcal{X}^{\star}$
\item if $y^{\star} \in \partial \mathcal{B}(C,R)$ then $y^{\star} \in \mathcal{P}$ hence $y^{\star} \in \mathcal{X}^{\star}$. 
\end{enumerate}
\end{proof}

In general, the radius $R_0$ is not known apriori. However, it is interesting to note that the construction of the centers of the intersection of balls does not need information about $R_0$. That information is only needed for computing the radius of the balls. As such, the centers of $\mathcal{Q}_{R_0^2}^k$ can be obtained, without knowing $R_0$. 

 The above Corollary \ref{Co2.1} can still be used in practice assuming that one knows a range for $R_0$ as $R_0 \in [\underline{R}_0, \overline{R}_0]$. 

We finally, give our last theorem of this paper. It computs the upper bound as a fixed point to a certain univariate function:

\begin{theorem}
Let $R_1 \in [\underline{R}_0, \overline{R}_0]$ such that 

\begin{align}\label{E2.32}
R_1 = \max_{x \in \mathcal{Q}_{R_1^2}^k} \|x - C_0\|
\end{align} Then 

\begin{align}
R_0 = \max_{x \in \mathcal{P}} \|x - C_0\| \leq R_1
\end{align}
\end{theorem}
\begin{proof}
Assume that
\begin{align}
R_0 = \max_{x \in \mathcal{P}} \|x - C_0\| = \max_{x \in \mathcal{Q}_{R_0^2}^k} \|x - C_0\|
\end{align} whatever value $R_0$ might have. In this situation one can actually compute $R_0$. 

Before going further, note that from the definition of $\mathcal{P}_{R^2}^{k}$ one has that the facets of this polytope have the following equations, see \cite{funcos1}:
\begin{align}
(C_0 - C_i)^T\cdot x + \|C_i\|^2 - r_i^2 - \|C_0\|^2 \leq -R^2
\end{align} hence the following monotony property holds: $\mathcal{P}_{R_a^2}^{k} \subseteq \mathcal{P}_{R_b^2}^{k}$ for $R_b \leq R_a$. Since each facet of the smaller polytope is a parallel translation of the bigger one, follows that the same stays true for the associated intersection of balls i.e. $\mathcal{Q}_{R_a^2}^{k} \subseteq \mathcal{Q}_{R_b^2}^{k}$ for $R_b \leq R_a$.

Take, $R \in  [\underline{R}_0, \overline{R}_0]$ and since $C_0 \not\in \text{conv}(C_1, \hdots, C_m)$ one can compute $\max_{x \in \mathcal{Q}_{R^2}^k} \|x - C_0\|$.

\begin{enumerate} 
\item  If $R < R_0$ then $\mathcal{Q}_{R_0^2}^k \subseteq \mathcal{Q}_{R^2}^k$ hence $\max_{x \in \mathcal{Q}_{R^2}^k} \|x - C_0\| > R_0 > R$. Increase $R$. 
\item If $R > R_0$  then $\mathcal{Q}_{R^2}^k \subseteq \mathcal{Q}_{R_0^2}^k$ hence $\max_{x \in \mathcal{Q}_{R^2}^k} \|x - C_0\| < R_0 < R$. Decrease $R$. 
\item If $R = R_0$ one expects $R = \max_{x \in \mathcal{Q}_{R^2}^k} \|x - C_0\|$.
\end{enumerate} In this situation, as can be observed, one can compute $R_0$ simply as the element $R$ of $[\underline{R}_0, \overline{R}_0]$ which meets

\begin{align}
R = \max_{x \in \mathcal{Q}_{R^2}^k} \|x - C_0\|
\end{align}

On the other hand, assume that 
 \begin{align}
R_0 = \max_{x \in \mathcal{P}} \|x - C_0\| < \max_{x \in \mathcal{Q}_{R_0^2}^k} \|x - C_0\|
\end{align} for the real (unknown) value of $R_0$. Still in this situation, compute $R_1 \in [\underline{R}_0, \overline{R}_0]$ with
\begin{align}
R_1 = \max_{x \in \mathcal{Q}_{R_1^2}^k} \|x - C_0\|
\end{align} 

Assuming that $R_1 < R_0$ follows that $\mathcal{Q}_{R_0^2}^k \subseteq \mathcal{Q}_{R_1^2}^k$, therefore 
\begin{align}
R_1 = \max_{x \in \mathcal{Q}_{R_1^2}^k} \|x - C_0\| &\geq \max_{ x \in \mathcal{Q}_{R_0^2}^k} \|x - C_0\| > \nonumber \\
& > \max_{x \in \mathcal{P}} \|x - C_0\|  = R_0
\end{align} which is a contradiction. 
\end{proof}

\begin{remark}
Note that one can always obtain a lower bound on $\max_{x \in \mathcal{P}} \|x - C_0\|$ as $\|x_0 - C_0\|$ for any point $x_0 \in \mathcal{P}$. 

To make the lower bound tight, we propose  to solve the linear problem: 
\begin{align}
\mathop{\text{argmax  }}_{x \in \mathcal{P}} v^T\cdot (x - C_0)
\end{align} for $v = \frac{y - C_0}{\|y - C_0\|}$ with $y \in \mathop{\text{argmax  }}_{x \in \mathcal{Q}_{R_1}^k} \|x - C_0\|$ where $R_1$ is given by (\ref{E2.32}).

\end{remark}

\section{Numerical results}

We implemented some parts of the described method and tested it on the problem of maximizing the distance over the unit hypercube to a random point inside the unit hypercube. We chose this problem, because it allows us to test the quality of the returned solution since for this particular case, the maximizer can be computed in polynomial time by a known algorithm. 

In the Figure \ref{fig3} one sees a 2-d example. One can see with light blue the hypercube, i.e. the polytope $\mathcal{P}$, with grey the polytope $\mathcal{P}_{R_0^2}^1$ and with green the polytope $\mathcal{P}_{R_0^2}^k$.

\begin{figure}[h]
  \includegraphics[width = 8cm]{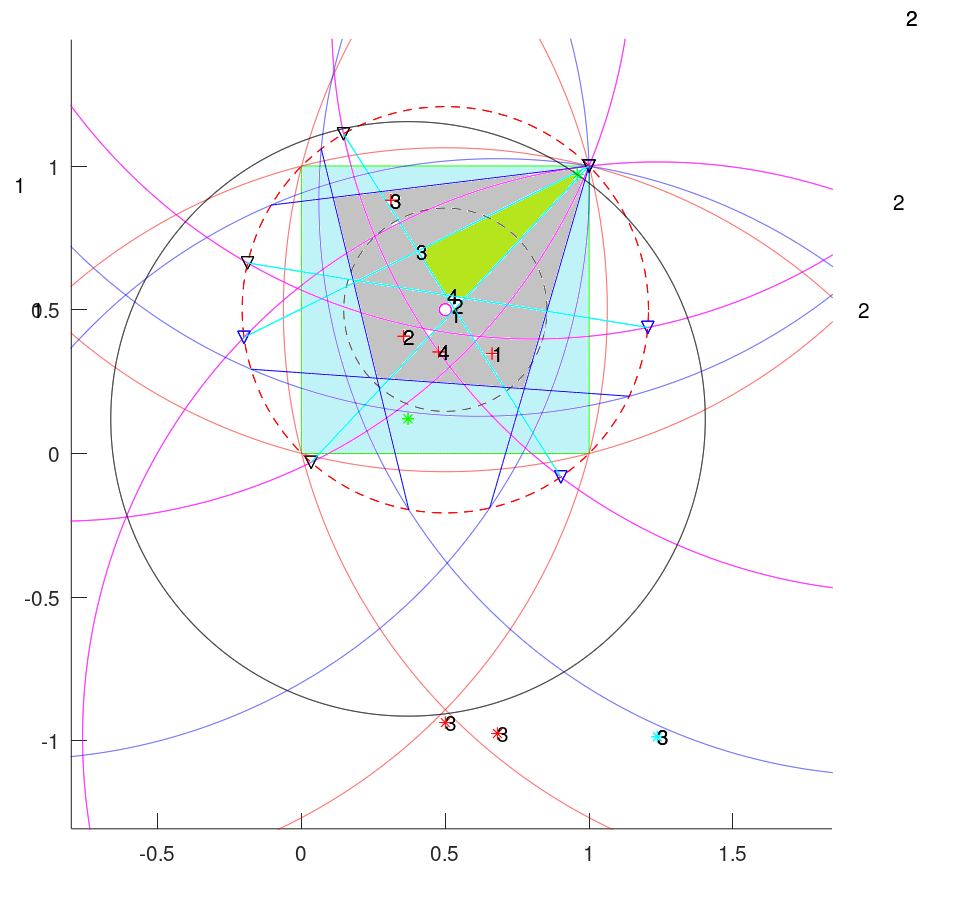}
  \caption{ With green one sees the last polytope and its associated intersection of balls. With green star the point $C_0$ and the black circle show the identified maximum distance. }
  \label{fig3}
\end{figure}

 In the Figure \ref{fig4} one sees another 2-d example. Here the original unit cube was randomly perturbed. One can see with green the perturbed hypercube, i.e. the polytope $\mathcal{P}$, with purple the polytope $\mathcal{P}_{R_0^2}^1$ and with blue the polytope $\mathcal{P}_{R_0^2}^k$. Note that the centers of the balls forming the intersection of balls $\mathcal{Q}_{R_0^2}^k$ associated to $\mathcal{P}_{R_0^2}^k$ does not contain $C_0$ in their convex hull.

\begin{figure}[h]
  \includegraphics[width = 8cm]{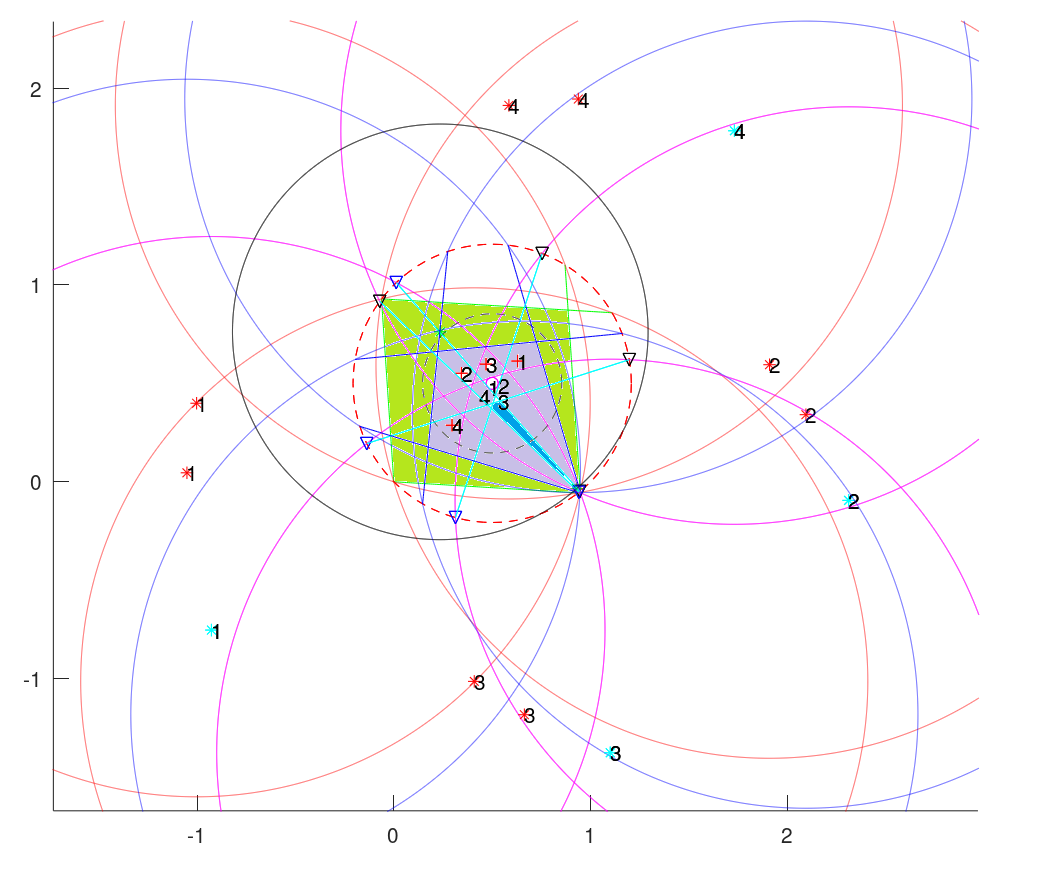}
  \caption{ A perturbed cube example. Note that the centers of the balls (with green stars) forming the intersection of balls $\mathcal{Q}_{R_0^2}^k$ associated to $\mathcal{P}_{R_0^2}^k$ (with blue) do not contain $C_0$ in their convex hull.}
  \label{fig4}
\end{figure}

Finally, we performed a test with higher dimensions. For $n = 100$ we obtained our solution and compared it with the correct solution. In Figure \ref{fig5} one can see the difference between their entries.   

\begin{figure}[h]
  \includegraphics[width = 8cm]{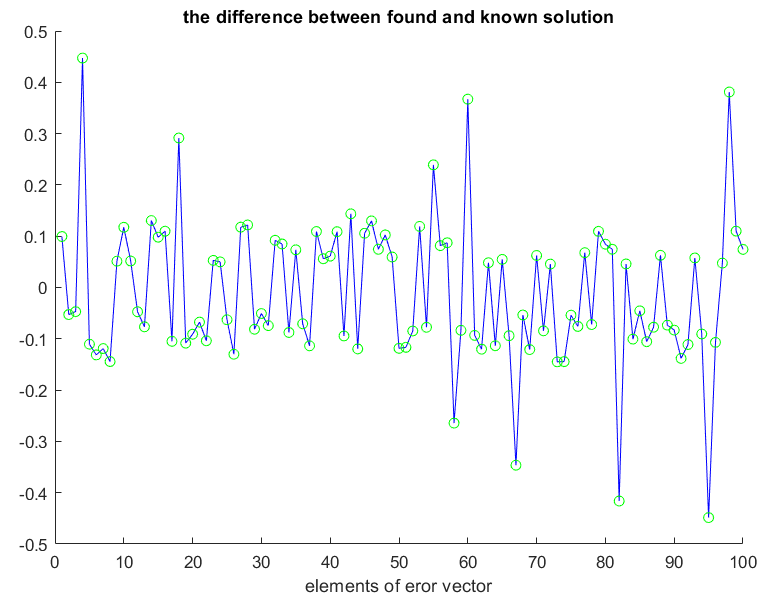}
  \caption{ The obtained error for each entry in the obtained solution. The dimension here is $n  = 100$. Our implementation uses the ellipsoid algorithm pedagogically implemented locally for some sub-procedures. Improvements can be made. }
  \label{fig5}
\end{figure}

\section{Conclusion}
This paper we analyzed the problem of maximizing the distance to a given point $C_0$ over a polytope $\mathcal{P}$, under the aiding hypothesis that a circumscribing ball is known whose boundary contain the vertices of $\mathcal{P}$ the farthest to $C_0$. As such we construct an arbitrary good approximation  of the polytope with an intersection of balls which preserves the vertices of the polytope $\mathcal{P}$ on the boundary of the circumscribing ball. 

Some results regarding the maximization of distances over intersection of balls are now used from \cite{funcos1} and \cite{funcos2} to motivate the construction of a new polytope which preserves the maximizers of the previous problem. From this polytope we construct again an intersection of ball and the process continues. We present a lemma which shown, that this process will eventually stop, in a finite number of steps, with an intersection of balls over which we can maximize the distance to the given point $C_0$. We show that this obtained distance is a nontrivial upper bound of the solution to the initial problem. We say, nontrivial since a trivial bound to the maximum distance over $\mathcal{P}$ to $C_0$ can be obtained ba simply maximizing the distance to $C_0$ over the ball which circumscribes $\mathcal{P}$. We also presented numerical results: several pedagogical 2-d examples and one instance in which a problem in dimension $n = 100$ was solved. The used problem for the numerical results is maximizing the distance over the unit hypercub to a random point inside the unit hypercube. We chose  this problem because it is possible to compute the distance by a known method and as such have something to compare our result against. 

\section{Future Work: a sequence of balls with decreasing radius }
As a continuation of this work, we propose the following. Given a polytope $\mathcal{P}$, a point $C_{0,1} \in \mathbb{R}^n$ and $R>0$ with $\mathcal{P} \subseteq \mathcal{B}(C_{0,1},R)$ one should investigate the possibility of obtaining a sequence of balls with decreasing radius containing $\mathcal{P}$. We manage to do so as follows: choose randomly $C_{0,2} \in \partial \mathcal{B}(C_{0,1},\rho)$ for some $\rho > 0$. 

From above, given a polytope included in a ball, note the possibility of obtaining an intersection of balls with the method described above (each ball, in the intersection of balls, boundary leaves the same imprint on the given circumscribing ball as a hyperplane defining the polytope facet). Denote by $Q$ the intersection of balls associated to $\mathcal{P}$. Note that:
\begin{enumerate}
\item unless $\mathcal{P}$ has a vertex on $\mathcal{B}(C_{0,1},R)$, then $\mathcal{P} \subsetneq \mathcal{Q} $ and the intersection of their ($\mathcal{P}$ and $\mathcal{Q}$) boundaries is empty.
\item the set $\mathcal{Q} \subseteq \mathcal{B}(C_{0,1},R)$ and their ($\mathcal{B}(C_{0,1},R)$ and $\mathcal{Q}$) boundaries do not have common elements.
\end{enumerate}

 We form the following sequence. Assuming $C_{0,2}$ is in the convex hull of the balls centers forming $\mathcal{Q}$ let 
\begin{align}
R_1 = \max_{x \in \mathcal{Q}} \|x - C_{0,2}\|
\end{align} and define $\mathcal{P}_{R_1^2}$ given by Theorem 1 in \cite{funcos1}. It is known that $\mathcal{P}_{R_1^2} \subseteq \mathcal{Q}$ and $\max_{x \in \mathcal{P}_{R_1^2}} \|x - C_{0,2}\| = \max_{x \in \mathcal{Q}} \|x - C_{0,2}\|$. Since $\mathcal{P}_{R_1^2} \subseteq \mathcal{B}(C_{0,1},R)$ associate to $\mathcal{P}_{R_1^2}$ an intersection of balls (with the explained method) $\mathcal{Q}_{R_1^2}^1$ then define $R_2 = \max_{x \in \mathcal{Q}_{R_1^2}^1} \|x - C_{0,2}\|$. Note that 
\begin{align}
R_1 = \max_{\mathcal{Q}}\|x - C_{0,2}\| &= \max_{\mathcal{P}_{R_1}^2} \|x - C_{0,2}\|\leq \nonumber \\
& \leq \max_{\mathcal{Q}_{R_1^2}^1}\|x - C_{0,2}\| = R_2
\end{align} As such, repeating the process as long as $C_{0,2}$ is in the convex hull of the balls forming $\mathcal{Q}_{R_{i}^2}^i$, one gets the sequence:
\begin{align}
R_k = \max_{\mathcal{Q}_{R_{k-1}^2}^{k-1}} \|x - C_{0,2}\| 
\end{align} with 
\begin{align}
R_1 \leq R_2 \leq \hdots \leq R_k
\end{align} According to Lemma \ref{C1} the process stops after a finite number of steps, say $k-1$. Therefore $C_{0,2}$ is no longer in the convex hull of the centers of the balls forming $\mathcal{Q}_{R_{k-1}^2}^{k-1}$ and one can compute $R_k$.  Let 
\begin{align}
\tilde{R} = \max_{x \in \mathcal{Q}_{\tilde{R}^2}^{k-1}} \|x - C_{0,2}\| 
\end{align} then $R_{k-1} < \tilde{R} < R_k$. Indeed, assuming $\tilde{R} \geq R_k$ then $R_k \leq \tilde{R} = \max_{x \in \mathcal{Q}_{\tilde{R}^2}^{k-1}}\|x - C_{0,2}\| <  \max_{x \in \mathcal{Q}_{R^2_{k-1}}^{k-1}}\|x - C_{0,2}\| = R_{k-1}<R_k$, which is a contradiction. Note the inequality is justified by the fact that since $\tilde{R} \geq R_k > R_{k-1}$ follows that $\mathcal{Q}_{\tilde{R}^2}^{k-1}  \subseteq  \mathcal{Q}_{R_{k-1}^2}^{k-1} $. Finally, $\tilde{R} \geq R_{k-1}$ because, otherwise, assuming that $\tilde{R} \leq R_{k-1}$ follows that $\mathcal{Q}_{R^2_{k-1}}^{k-1} \subseteq \mathcal{Q}_{\tilde{R}^2}^{k-1} $ hence 
\begin{align}
R_{k-1} > \tilde{R} = \max_{\mathcal{Q}_{\tilde{R}^2}^{k-1}} \|x - C_{0,2}\| >\max_{\mathcal{Q}_{\tilde{R}^2}^{k-1}} \|x - C_{0,2}\|  = R_{k} 
\end{align} which is a contradiction.

As such, since $R_0 := \max_{\mathcal{P}}\|x - C_{0,2}\| \leq \max_{\mathcal{Q}}\|x - C_{0,2}\| < \hdots < R_{k-1} \leq \tilde{R}$), now consider the ball $\mathcal{B}(C_{0,2},\tilde{R})$ with $\mathcal{P} \subseteq \mathcal{B}(C_{0,2},\tilde{R})$. 

Now, we have 
\begin{align}
\mathcal{P} \subseteq \mathcal{B}(C_{0,1},R) \cap \mathcal{B}(C_{0,2}, \tilde{R}) \subseteq \mathcal{B}(C_{0,1}^1, R_{1,1})
\end{align} where $\mathcal{B}(C_{0,1}^1, R_{1,1})$ is the smallest ball including $ \mathcal{B}(C_{0,1},R) \cap \mathcal{B}(C_{0,2}, \tilde{R})$. It is obvious that $R_{1,1} \leq R$. 

Letting $C_{0,2}^1 \in \partial \mathcal{B}(C_{0,1}^1,\rho)$ one can repeat the process with the circumscribing ball $\mathcal{B}(C_{0,1}^1,R_{1,1})$ to obtain
\begin{align}
\mathcal{P} \subseteq \mathcal{B}(C_{0,1}^1,R_{1,1}) \cap \mathcal{B}(C_{0,2}^1, \tilde{R}_{1}) \subseteq \mathcal{B}(C_{0,1}^2, R_{1,2})
\end{align} with $R_{1,2} < R_{1,1} < R$. As such, one obtains a the sequence of balls with the property:
 \begin{align}
\mathcal{P} \subseteq \mathcal{B}(C_{0,1}^p,R_{1,p}) \cap \mathcal{B}(C_{0,2}^p, \tilde{R}_{p}) \subseteq \mathcal{B}(C_{0,1}^{p+1}, R_{1,p+1})
\end{align}  with 
\begin{align}
R_{1,p} \leq R_{1,p-1} \leq \hdots \leq R_{1,1} \leq R
\end{align}

Since this sequence is monotonically decreasing and bounded below, if follows that it has a limit, say $R_3$. It is interesting to analyze the properties of this sequence of balls.

\end{document}